\documentclass{amsart}
\usepackage{enumerate}
\usepackage{hyperref}
\usepackage{amsthm,tikz,subfig}
\usetikzlibrary{shadows}
\hypersetup{
    pdftitle=   {Even-cycle decompositions of graphs with no odd-K4-minor},
   pdfauthor=  {Tony Huynh,  Sang-il Oum, Maryam Verdian-Rizi}
}
\newtheorem{THM}{Theorem}[section]
\newtheorem*{THMMAIN}{Theorem~\ref{thm:main-signed}}

\newtheorem{LEM}[THM]{Lemma}
\newtheorem{PROP}[THM]{Proposition}

\theoremstyle{remark}
\newcommand\graph{\mathsf{graph}}
\newcommand\abs[1]{\lvert #1\rvert}

\begin{document}
\title{Even-cycle decompositions of graphs with no odd-$K_4$-minor}
\author{Tony Huynh}
\author{Sang-il Oum}
\author{Maryam Verdian-Rizi}
\address[Huynh]{Department of Mathematics,
  Universit\'e Libre de Bruxelles, Brussels, Belgium}
\address[Oum]{Department of Mathematical Sciences, KAIST, Daejeon, South Korea and School of Mathematics, KIAS, Seoul, South Korea}
\address[Verdian-Rizi]{Department of Mathematical Sciences, KAIST, Daejeon, South Korea}
 \email{tony.bourbaki@gmail.com}
\email{sangil@kaist.edu}
\email{mverdian@gmail.com}
\thanks{Supported by Basic Science Research
  Program through the National Research Foundation of Korea (NRF)
  funded by the Ministry of Education, Science and Technology
  (2012-0004119). S. O. is also supported by TJ Park Junior Faculty Fellowship.}
\thanks{T.~H., S.~O., and M.~V.-R. are supported by Basic Science Research
  Program through the National Research Foundation of Korea (NRF)
  funded by the Ministry of Science, ICT \& Future Planning
  (2011-0011653).  T.~H. was also supported by the NWO (The Netherlands Organization for Scientific Research) free
competition project ``Matroid Structure - for Efficiency'' led by Bert Gerards}
\date{\today}

\begin{abstract}
An \emph{even-cycle decomposition} of a graph $G$ is a partition of
$E(G)$ into cycles of even length.  Evidently, every Eulerian bipartite graph has an even-cycle decomposition.  Seymour (1981) proved that
every $2$-connected loopless Eulerian planar graph with an even number
of edges also admits an even-cycle decomposition.  Later, Zhang (1994) generalized this to graphs
with no $K_5$-minor.

Our main theorem gives sufficient conditions for the existence of even-cycle decompositions of  graphs in the absence of \emph{odd minors}. Namely, we prove that every $2$-connected loopless
Eulerian odd-$K_4$-minor-free graph with an even number of edges has an even-cycle decomposition.

This is best possible in the sense that `odd-$K_4$-minor-free' cannot be replaced with 
`odd-$K_5$-minor-free.'  The main technical ingredient is a structural characterization of the class of 
odd-$K_4$-minor-free graphs, which is due to Lov\'{a}sz, Seymour, Schrijver, and Truemper. 
\end{abstract}

\keywords{cycle, even-cycle decomposition, Eulerian graph, odd minor}

\maketitle

\section{Introduction}\label{sec:intro}
A graph $G$ is \emph{even-cycle decomposable} if its edge set can be partitioned into even cycles.  Note that if $G$ is even-cycle decomposable, then necessarily $G$ is Eulerian, loopless, and $\abs{E(G)}$ is even.  For bipartite graphs, these conditions are also sufficient, since every cycle is even.

\begin{PROP}[Euler] \label{trivial}
  Every Eulerian bipartite graph is even-cycle decomposable.
\end{PROP}

One motivation to study even-cycle decompositions is their connection to the four colour theorem~\cite{AH89, RSST96}.  For example, as noted by Seymour~\cite{planar}, one consequence of the four colour theorem is that  every $2$-connected cubic planar graph has a set of \emph{even} cycles in which each edge occurs exactly twice.  In the same paper, Seymour also proves that planar graphs (which satisfy the obvious necessary conditions) are always even-cycle decomposable.

\begin{THM}[Seymour~\cite{planar}]\label{seymour}
  Every $2$-connected Eulerian loopless planar graph with an even number of edges is even-cycle decomposable.
\end{THM}

Note that the $2$-connected condition is with little loss of generality, since a graph $G$ is even-cycle decomposable if and only if each block of $G$ is even-cycle decomposable.

Later, Zhang~\cite{nok5} generalized Theorem~\ref{seymour} to graphs with no $K_5$-minor.

\begin{THM}[Zhang~\cite{nok5}]\label{zhang}
Every $2$-connected Eulerian loopless $K_5$-minor-free graph with an even number of edges is even-cycle decomposable.
\end{THM}

In this paper we consider sufficient conditions for the existence of even-cycle decompositions in graphs with no \emph{odd-$K_t$-minor} (definitions are deferred until the next section).
For further information on even-cycle decomposable graphs and related results we refer the reader to the surveys of Jackson~\cite{jackson} or Fleischner~\cite{fleischner}.

Our main result is the following.

\begin{THM} \label{oddk4free}
Every $2$-connected Eulerian loopless odd-$K_4$-minor-free graph with an even number of edges is even-cycle decomposable.
\end{THM}

Theorem~\ref{oddk4free} is best possible in the sense that `odd-$K_4$-minor-free' cannot be replaced with 
`odd-$K_5$-minor-free.'

\begin{THM} \label{counterexample}
There exists a $2$-connected Eulerian loopless odd-$K_5$-minor-free graph with an even number of edges which is not even-cycle decomposable.
\end{THM}

In a previous version of this paper, we conjectured that all $2$-connected Eulerian loopless odd-$K_5$-minor-free graphs with an even number of edges are even-cycle decomposable.
Note that if true, this conjecture would simultaneously imply both Proposition~\ref{trivial} and Theorem~\ref{zhang}.  By a celebrated theorem of Guenin~\cite{Guenin01}, a signed graph is odd-$K_5$-minor-free if and only if it is weakly bipartite (as defined by Gr{\"o}tschel and Pulleyblank~\cite{GP81}).  See also Naserasr, Rollov\'a, and Sopena~\cite{NRS13} for connections between  odd-$K_5$-minor-free signed graphs and the $4$-colour theorem.

However, Cheolwon Heo disproved our conjecture.  We would like to thank him for graciously allowing us to include the proof of Theorem~\ref{counterexample} in this paper.  The (graph version of the) example used in the proof of Theorem~\ref{counterexample} also appears in~\cite{planar}, where it is used as a counterexample to the claim that all $2$-connected Eulerian loopless graphs with an even number of edges is even-cycle decomposable.  Note that $K_5$ is another counterexample to this claim.

For the proof of Thereom~\ref{oddk4free}, a potentially useful inductive tool is the following nice theorem of Conlon~\cite{conlon}.

\begin{THM}[Conlon~\cite{conlon}] \label{Conlon}
  Let $G$ be a simple $3$-connected graph of minimum degree $4$.
  If $G$ is not isomorphic to $K_5$,
  then $G$ contains an even cycle $C$ such that $G \setminus E(C)$ is $2$-connected.
\end{THM}

Unfortunately the assumptions in Theorem~\ref{Conlon} are much too strong for our purposes.  Instead, our result relies on a structural description of the signed graphs with no odd-$K_4$-minor.  According to Gerards~\cite{gerards}, this structure theorem is due to Lov\'{a}sz, Seymour, Schrijver, and Truemper.  The proof of the structure theorem makes use of the regular matroid decomposition theorem of Seymour~\cite{regular}.  It is also a special instance of a decomposition theorem for binary matroids with no $F_7$-minor using a fixed element due to Truemper and Tseng~\cite{TT1986,nofano} ($F_7$ denotes the Fano matroid).

Our strategy for proving Theorem~\ref{oddk4free} is to first prove it for the two basic classes given by the structure theorem.  We then show how to combine even-cycle decompositions across low order separations.  This last step contains some technical difficulties.  Indeed, we end up having to apply the structure theorem \emph{twice} (once on the original graph and then also on an auxiliary graph).  

The graphs we consider are not necessary simple, and we do not know an easy proof of our theorem for simple graphs.  For example, the decomposition step will introduce parallel edges even if the original graph is simple. Our proof would be slightly shorter if we used Theorem~\ref{seymour}, but we do a bit of extra work so that we avoid using Theorem~\ref{seymour}.

The rest of the paper is organized as follows.  In
Section~\ref{sec:sign}, we define signed graphs and prove Theorem~\ref{counterexample}.  In
Section~\ref{sec:structure}, we present the structure theorem for
signed graphs with no odd-$K_4$-minor.  In Section~\ref{sec:bipartite}
and Section~\ref{sec:planar} we prove our main theorem for the two basic 
classes of signed graphs appearing in the structure theorem.  Finally, in Section~\ref{sec:proof}, we prove our main theorem.

\section{Signed graphs, re-signing, and odd minors}\label{sec:sign}
A \emph{signed graph} is a pair $(G, \Sigma)$ consisting of a graph
$G$ together with a \emph{signature} $\Sigma \subseteq E(G)$.  The
edges in $\Sigma$ are \emph{negative} and the other edges are \emph{positive}.
A cycle (or path) is \emph{balanced} if it contains an even number of negative
edges; otherwise it is \emph{unbalanced}.  We say
that a signed graph $(G, \Sigma)$ is \emph{balanced-cycle decomposable},
if $E(G)$ can be partitioned into balanced cycles of  $(G, \Sigma)$.

For a signed graph $(H, \Sigma)$ define $\graph(H,\Sigma)$ to be the
graph obtained from $H$ by subdividing every positive edge once. 
Then it is easy to observe the following lemma, because a cycle in $(H, \Sigma)$ is balanced if and only if the corresponding cycle in $\graph(H,\Sigma)$ is even. This lemma will be used later as we will frequently reduce signed graphs to graphs.
\begin{LEM}\label{lem:signed}
A signed graph $(H, \Sigma)$ is balanced-cycle decomposable if and only if
$\graph(H, \Sigma)$ is even-cycle decomposable.
\end{LEM}
For $X \subseteq V(G)$, we let $\delta_G(X)$ be the set of edges with exactly one end in $X$.  We say that $\delta_G(X)$ is the \emph{cut induced by $X$}.  Two signatures $\Sigma_1, \Sigma_2 \subseteq E(G)$ are \emph{equivalent} if their symmetric difference is a cut.  The operation of changing to an equivalent signature is called \emph{re-signing}.  A key observation is that if $\Sigma_1, \Sigma_2 \subseteq E(G)$ are equivalent signatures, then $(G, \Sigma_1)$ and $(G, \Sigma_2)$ have exactly the same set of balanced cycles.  Thus, for equivalent signatures $\Sigma_1$ and $\Sigma_2$, $(G, \Sigma_1)$ is balanced-cycle decomposable if and only if $(G, \Sigma_2)$ is balanced-cycle decomposable.

We will require the following well-known lemma, first proved by Zaslavsky~\cite{zaslavsky82}.

\begin{LEM}[Zaslavsky~\cite{zaslavsky82}] \label{resign}
Let $(G, \Sigma)$ be a signed graph.  For every forest $F$ which is a subgraph of $G$, there exists a signature which is disjoint from $E(F)$ and equivalent to $\Sigma$.
\end{LEM}

A \emph{minor} of a signed graph $(G, \Sigma)$ is a signed graph
that can be obtained from $(G,\Sigma)$ by any of the following
operations: re-signing, deleting edges or vertices, and contracting
positive edges.  For a graph $H$,  \emph{odd-$H$} is the signed graph
$(H, E(H))$.
A signed graph is \emph{odd-$H$-minor-free} if it has no minor that is
isomorphic to an odd-$H$.
A graph $G$ is \emph{odd-$H$-minor-free} if the odd-$G$ is odd-$H$-minor-free.

Now that our terms have been defined, we restate our main theorem.
\begin{THM}\label{thm:main-signed}
Every $2$-connected Eulerian loopless odd-$K_4$-minor-free signed graph with an even number of negative edges is balanced-cycle decomposable.
\end{THM}

As we discussed in Section~\ref{sec:intro}, $K_4$ cannot be replaced with $K_5$ in the above theorem. For that, 
we state an equivalent formulation of  Theorem~\ref{counterexample} and present a proof which is due to Cheolwon Heo.

\begin{THM}
There exists a $2$-connected Eulerian loopless odd-$K_5$-minor-free signed graph with an even number of negative edges which is not balanced-cycle decomposable.
\end{THM}

\begin{figure}
  \centering
  \tikzstyle{every node}=[circle,draw,fill=black!50,inner
  sep=0pt,minimum width=4pt]
  \begin{tikzpicture}
  \foreach \x in {0} {
  \draw (90+72*\x:2) node [label=$\x$] (w\x) {};
  }
  \foreach \x in {1} {
  \draw (90+72*\x:2) node [label=left:$\x$] (w\x) {};
  }
  \foreach \x in {4} {
  \draw (90+72*\x:2) node [label=right:$\x$] (w\x) {};
  }
  \foreach \x in {2,3} {
  \draw (90+72*\x:2) node [label=below:$\x$] (w\x) {};
  }
  \foreach \x/\y in {1/6,4/9} {
  \draw (90+72*\x:1) node [label=$\y$] (v\x) {};
  }
  \foreach \x/\y in {2/7} {
  \draw (90+72*\x:1) node [label=above left:$\y$] (v\x) {};
  }
    \foreach \x/\y in {0/5} {
  \draw (90+72*\x:1) node [label=left:$\y$] (v\x) {};
  }
  \foreach \x/\y in {3/8} {
  \draw (90+72*\x:1) node [label=above right:$\y$] (v\x) {};
  }
  \foreach \x in {0,1,2,3,4}{
    \draw [bend left] (v\x) to (w\x);
    \draw [dashed,bend right] (v\x) to (w\x);
  }
  \draw (v1)--(v3)--(v0)--(v2)--(v4)--(v1);
  \draw [dashed] (w0)--(w1)--(w2)--(w3)--(w4)--(w0);
  \end{tikzpicture}
  \caption{A $2$-connected Eulerian signed graph with no odd-$K_5$-minor and no balanced-cycle decompositions. (Solid lines denote negative edges.)}
  \label{fig:counterex}
\end{figure}
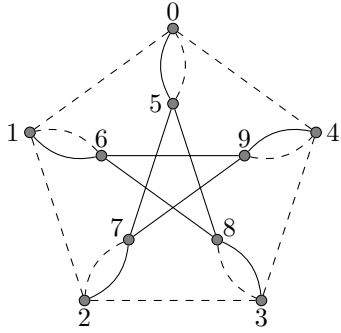
\begin{proof}
  We claim that the signed graph $(G, \Sigma)$ in Figure~\ref{fig:counterex} is such a signed graph.  Evidently, $G$ is $2$-connected, Eulerian and loopless.  We claim that $(G, \Sigma)$ is not balanced-cycle decomposable. Let $C$ be an arbitrary cycle in a balanced-cycle decomposition of $(G, \Sigma)$.  Note that $G$ only contains cycles of length $2,5,6,8$, or $9$.  Since all $2$-cycles are unbalanced, $C$ has length $5,6,8$, or $9$.  On the other hand, it is easy to check that if $C$ has length $5$ or $6$, then some block of $(G \setminus E(C), \Sigma \setminus E(C))$ is an unbalanced $2$-cycle, which is a contradiction.  Thus, $C$ has length $8$ or $9$.  But $|E(G)|=20$ and $8a+9b \neq 20$ for all non-negative integers $a,b$.  
  
  It remains to show that $(G, \Sigma)$ does not contain an odd-$K_5$-minor.  By degree considerations, the only way to obtain a $K_5$-minor from the Petersen graph is to contract a perfect matching.  Thus, to obtain odd-$K_5$ from $(G,\Sigma)$ we must delete exactly one edge from each $2$-cycle and then contract a perfect matching $M$.  Let $C$ be the balanced $5$-cycle $57968$.  It is easy to see that $|M \cap E(C)| \in \{0,2\}$. Thus, our odd-$K_5$-minor either contains a balanced $5$-cycle or a balanced $3$-cycle, which is a contradiction. 
\end{proof}

\section{Structure theorem for signed graphs with no odd-$K_4$-minor} \label{sec:structure}
In this section we describe the structure of signed graphs with no odd-$K_4$-minor.  We begin by describing the basic building blocks in the decomposition theorem.

\subsection*{Almost balanced.} A signed graph is
\emph{balanced} if every cycle is balanced.
Note that Lemma~\ref{resign} easily implies the following alternate definition of balanced signed graphs.
\begin{LEM}\label{bipartite}
  A signed graph is balanced if and only if we can re-sign so that all its edges are positive.
\end{LEM}

Since re-signing an Eulerian signed graph $(G,\Sigma)$ does not change
the parity of~$\abs{\Sigma}$, it follows that balanced Eulerian
signed graphs always contain an even number of negative edges.  We say that
a signed graph is \emph{almost balanced} if there exists a vertex
whose deletion yields a balanced signed graph.  Since balanced
signed graphs have no odd-$K_3$-minors, it follows that almost
balanced signed graphs have no odd-$K_4$-minors.

\subsection*{Planar with two unbalanced faces.} A signed graph $(G, \Sigma)$
is \emph{planar} if the underlying graph $G$ is planar.  A face $F$ of
a planar embedding of $(G, \Sigma)$ is \emph{balanced} if the facial walk
corresponding to $F$ contains an even number of negative edges, otherwise
$F$ is \emph{unbalanced}.  We say that $(G, \Sigma)$ is \emph{planar with at
  most two unbalanced faces} if $(G, \Sigma)$ has a planar embedding with at
most two unbalanced faces.
Notice that if every face of a planar embedding of $(G,\Sigma)$ is balanced, then $(G,\Sigma)$ is balanced.

Observe that the property of being planar with at most two unbalanced faces is
preserved under taking minors and that odd-$K_4$ does not have this
property.  Therefore, signed graphs that are planar with at most two
unbalanced faces do not have odd-$K_4$-minors.

\subsection*{The signed graph $\tilde{K_{3}^{2}}$.}  We define
$\tilde{K_{3}^{2}}$ to be the signed graph $(G, \Sigma)$ where $G$ is
a triangle with doubled edges and $\Sigma$ is 
the edge-set of a triangle (see Figure~\ref{fig:k32}).   Evidently,
$\tilde{K_{3}^{2}}$ has no odd-$K_4$-minor, but it is neither almost
balanced nor planar with two unbalanced faces.  Two more small signed graphs
that turn up in our proofs are $\tilde
    K_2^2$ (two vertices connected by a positive and a negative edge) and $\tilde
    K_2^2 \cdot \tilde
    K_2^2$ (two $\tilde{K}_2^2$'s joined at a vertex).  See
    Figure~\ref{fig:k32} for pictures of $\tilde{K_{3}^{2}}, \tilde K_2^2 \cdot \tilde
    K_2^2$, and $\tilde
    K_2^2$.

\begin{figure}
  \centering
  \tikzstyle{every node}=[circle,draw,fill=black!50,inner
  sep=0pt,minimum width=4pt]
  \subfloat[$\tilde K_{3}^2$]{
  \begin{tikzpicture}
    \draw [bend left] (90:1) node (v1) {} to (210:1) node (v2) {} to (330:1) node (v3) {} to (v1);
    \draw [dashed,bend right] (v1) to (v2) to (v3) to (v1);
  \end{tikzpicture}}
  \subfloat[$\tilde
    K_2^2 \cdot \tilde K_2^2$]{\hspace{5em}%
  \begin{tikzpicture}
    \draw [bend left] (90:1) node (v1) {} to (90:0) node (v2) {} to (-90:1) node (v3) {};
    \draw [dashed,bend right] (v1) to (v2) to (v3);
  \end{tikzpicture}\hspace{5em}}
  \subfloat[$\tilde K_2^2$]{
  \begin{tikzpicture}
    \draw [bend left] (210:1) node (v1) {} to (330:1) node (v2) {} ;
    \draw [dashed] (v1) to [bend right] (v2) ;
  \end{tikzpicture}}
  \caption{The signed graphs $\tilde K_{3}^2$, $\tilde
    K_2^2 \cdot \tilde K_2^2$, and $\tilde K_2^2$. (Solid lines denote negative edges.)}
  \label{fig:k32}
\end{figure}
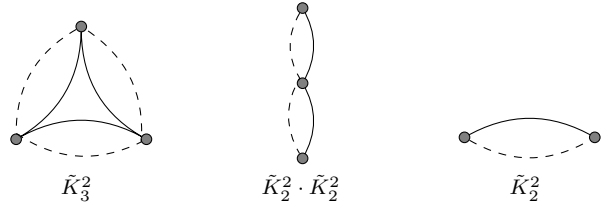

\medskip

These turn out to be essentially the only signed graphs with no odd-$K_4$-minor.  To define what we mean by `essentially', we introduce
the notion of separations.  A \emph{separation} of a graph
$G$ is a pair $(G_1,G_2)$ where $G_1$ and $G_2$ are edge-disjoint
subgraphs of $G$, such that $G_1 \cup G_2 = G$.  The
\emph{boundary} of $(G_1, G_2)$ is $V(G_1) \cap V(G_2)$, and its
\emph{order} is $\abs{V(G_1) \cap V(G_2)}$.  We say that $(G_1, G_2)$ is a
\emph{proper separation} if both $V(G_1) \setminus V(G_2)$ and $V(G_2)
\setminus V(G_1)$
are non-empty.  A separation of order $k$ is
called a \emph{$k$-separation}.  A \emph{separation} of a signed graph $(G, \Sigma)$ is simply a separation of $G$.  

We can now state the structure theorem.

\begin{THM}[{Lov\'asz, Seymour, Schrijver, and Truemper in Gerards~\cite[Theorem 3.2.4]{gerards}}] \label{structure}
Let $(G, \Sigma)$ be a signed graph with no odd-$K_4$-minor.  Then at
least one of the following holds.
\begin{enumerate}[(i)]
\item
$(G, \Sigma)$ is almost balanced, or planar with two unbalanced faces, or isomorphic to~$\tilde{K_{3}^{2}}$.

\item
$(G, \Sigma)$ is not $2$-connected.

\item
$(G, \Sigma)$ has a $2$-separation $(G_1, G_2)$ where each
$(G_i,E(G_i)\cap \Sigma)$ is
connected and not a signed subgraph of $\tilde{K_{2}^{2}}$.

\item
$(G, \Sigma)$ has a $3$-separation $(G_1, G_2)$ where $(G_2,E(G_2)\cap
\Sigma)$ is balanced, connected, and has at least $4$ edges.
\end{enumerate}
\end{THM}

\section{Almost balanced signed graphs} \label{sec:bipartite}
We begin by proving Theorem~\ref{thm:main-signed} for almost balanced signed graphs.
We require the following lemma.

\begin{LEM} \label{surprising}
If $G$ is a connected graph and $X$ is a set of $2k$ vertices of $G$, then there is a collection of $k$ pairwise edge-disjoint paths in $G$ whose set of ends is precisely $X$.
\end{LEM}

\begin{proof}
Let $(G,X)$ be a counterexample with $\abs{E(G)}$ minimum.  Note that $G$ must be a tree, since every spanning tree of $G$ is also a counterexample.  Next, observe that each leaf of $G$ is in $X$, otherwise deleting such a leaf contradicts minimality.  Let $\ell$ be a leaf and let $w$ be the unique neighbour of $\ell$.  If $w \in X$, then $(G-\ell, X \setminus \{\ell,w\})$ is a smaller counterexample since we can link $\ell$ and $w$ via the edge $\ell w$.  On the other hand, if $w \notin X$, then deleting $\ell$ and adding $w$ to $X$ yields a smaller counterexample.
\end{proof}
\begin{PROP} \label{almostbipartite}
Every $2$-connected Eulerian loopless almost balanced signed graph with an even number of negative edges is balanced-cycle decomposable.
\end{PROP}

\begin{proof}
Let $(H, \Sigma)$ be
a $2$-connected Eulerian loopless almost balanced signed graph with
an even number of negative edges.  Call a graph $G$ \emph{almost bipartite} if $G-v$ is bipartite for some $v \in V(G)$.  By replacing $(H, \Sigma)$ with
$\graph(H, \Sigma)$, it suffices to show that every $2$-connected
Eulerian loopless almost bipartite \emph{graph} with an even number of
edges is even-cycle decomposable.  Let $G$ be such a graph and let $v$
be a vertex of $G$ such that $G - v$ is bipartite.  We may assume that there are at most two parallel edges between every pair of vertices, else we can remove a $2$-cycle and apply induction.

Let $(A,B)$ be a bipartition of $G-v$.  Let $X$ be the set of
neighbours of $v$ in $A$ and partition $X$ as $X_1 \cup X_2$, where $x
\in X_i$ if and only if there are $i$ edges between $x$ and $v$.
Since $|E(G)|$ is even, it follows that $\abs{X_1}$ is even because
\[
\abs{E(G)}=\sum_{u \in B} \deg_{G} (u) + 2\abs{X_2} + \abs{X_1}.
\]

Now, since $G$ is $2$-connected, the graph $G-v$ is connected.
Therefore, by Lemma~\ref{surprising} there is a collection
$\mathcal{P}$ of $\abs{X_1}/2$ pairwise edge-disjoint paths in $G-v$
whose set of ends is precisely $X_1$.  Note that each path $P \in
\mathcal{P}$ has even length since $G-v$ is bipartite.  Thus we may
combine the paths in $\mathcal{P}$ with edges between $v$ and
$X_1$ to obtain a family $\mathcal{C}_1$ of $\abs{X_1}/2$ pairwise
edge-disjoint even cycles.  Evidently, the edges between $v$ and
$X_2$ can be partitioned into a family $\mathcal{C}_2$ of $2$-cycles.
Let $\mathcal{E}$ be the set of edges in $\mathcal{C}_1 \cup
\mathcal{C}_2$.  Observe that the graph $G - \mathcal{E}$ is bipartite
since in $G- \mathcal{E}$, the vertex $v$ is only adjacent to vertices in $B$.  Hence, $G- \mathcal{E}$ is even-cycle decomposable and we are done.
\end{proof}

\section{Planar signed graphs with two unbalanced faces} \label{sec:planar}
We now prove that planar signed graphs with two unbalanced faces are balanced-cycle decomposable.
Note that this follows from Theorem~\ref{seymour}, but we give a short
proof in order to keep our proof of Theorem~\ref{thm:main-signed} self-contained.

\begin{PROP} \label{oddplanar}
Every $2$-connected Eulerian loopless planar signed graph with an even number of negative edges and exactly two unbalanced faces is balanced-cycle decomposable.
\end{PROP}

\begin{proof}
 Let $(H,
\Sigma)$ be a $2$-connected Eulerian loopless planar signed graph with
an even number of negative edges and exactly two unbalanced faces.  By passing to
$\graph(H, \Sigma)$, it suffices to show that every $2$-connected
Eulerian loopless planar \emph{graph} with an even number of edges and
exactly two odd-length faces is even-cycle decomposable.  Let $G$ be such a
graph and let $F_1$ and $F_2$ be the two odd-length faces of $G$.
Since $G$ is Eulerian, the dual graph $G^*$  of $G$ is bipartite.  Let $(A,B)$ be a bipartition of $V(G^*)$.  Since $G$ (and hence also $G^*$) has an even number of edges, $F_1$ and $F_2$ must be on the same side of the bipartition, say $F_1, F_2 \in A$.  Since $B$ is both an independent set and a vertex cover of $G^*$, the faces in $G$ corresponding to the vertices in $B$ form an even-cycle decomposition of $G$.
\end{proof}

\section{Proof of Theorem~\ref{thm:main-signed}} \label{sec:proof}
In this section we prove Theorem~\ref{thm:main-signed}, thus proving our main theorem, Theorem~\ref{oddk4free}.  We start with a simple
parity lemma.

\begin{LEM}\label{parity}
  Let $G$ be an Eulerian graph and let $(G_1,G_2)$ be a
  separation. Then
  \[
  \sum_{v\in V(G_1)\cap V(G_2)} \deg_{G_1}(v)\equiv
  \sum_{v\in V(G_1)\cap V(G_2) }\deg_{G_2}(v)\equiv
      0\pmod 2.
  \]
\end{LEM}
\begin{proof}
  Observe that
  \begin{align*}
  \sum_{v\in V(G_1)}\deg_G (v)= 2 \abs{E(G_1)}+
  \sum_{v\in V(G_1)\cap V(G_2)}\deg_{G_2} (v).
  \end{align*}
  The lemma then follows easily from the above equation.
\end{proof}

A $2$-separation $(G_1,G_2)$ of an Eulerian graph $G$ is 
\emph{odd} if $\deg_{G_1} (v)$ is odd for all vertices $v\in V(G_1)\cap V(G_2)$.  It is \emph{even} if $\deg_{G_1} (v)$ is even for for all vertices $v\in V(G_1)\cap V(G_2)$. By Lemma~\ref{parity}, every $2$-separation of an Eulerian graph is either odd or even.  

Our next lemma
asserts that $2$-connected Eulerian graphs with at least one
even $2$-separation can be decomposed into a `necklace structure' of
Eulerian subgraphs.

\begin{LEM}\label{bead}
  Let $G$ be a $2$-connected Eulerian loopless graph having an
  even $2$-separation $(G_1,G_2)$ such that $G_1$ and $G_2$ are connected.
 Then  there exist pairwise
  edge-disjoint connected Eulerian subgraphs $B_1$, $B_2$, $\ldots$,
  $B_n$ of $G$  with $n\ge 2$
  satisfying the following.
  \begin{enumerate}[(i)]
  \item $\bigcup_{i=1}^n E(B_i)=E(G)$.
  \item For each $i$, either $B_i$ is $2$-connected or
    $B_i$ has two vertices.
  \item \begin{enumerate}
  \item
  Either $n=2$ and $(B_1,B_2)$ is a $2$-separation of $G$, 
  \item
  or $n \geq 3$ and for all $1\le i<j\le n$,
    \[\abs{V(B_i)\cap V(B_j)}=
    \begin{cases}
      1 & \text{if }i-j\equiv \pm1 \pmod n,\\
      0&\text{otherwise.}
    \end{cases}\] 
  \end{enumerate}
  \item There exists  $k$ such that
    $B_1\cup B_2\cup \cdots B_k=G_1$
    and $B_{k+1}\cup \cdots \cup B_n=G_2$.
  \end{enumerate}
\end{LEM}

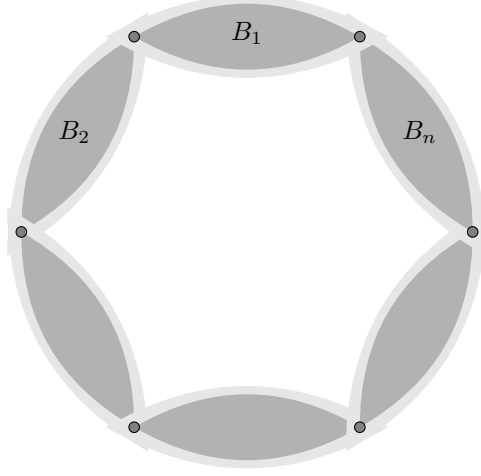
\begin{figure}
\centering
\tikzstyle{v}=[circle,draw,fill=black!50,inner
  sep=0pt,minimum width=4pt,solid,thin]
  \tikzstyle{b}=[fill=black!30, circular glow={fill=black!10},bend left]
\begin{tikzpicture}
\fill [b] (120:3) to (180:3) to (120:3);
\fill [b] (180:3) to (-120:3) to (180:3);
\fill [b] (-120:3) to (-60:3) to (-120:3);
\fill [b] (-60:3) to (0:3) to (-60:3);
\fill [b] (0:3) to (60:3) to (0:3);
\fill [b] (60:3) to (120:3) to (60:3);
\foreach \x in {0,60,120,180,-120,-60} {
\draw (\x:3) node [v]{};
}
\node at(90:2.65) {$B_1$};
\node at(30:2.65) {$B_n$};
\node at(150:2.65) {$B_2$};
\end{tikzpicture}
\caption{A \emph{necklace decomposition} into \emph{beads} $B_1, \dots, B_n$.}
\label{fig:beads}
\end{figure}
\begin{proof}
  We choose pairwise
  edge-disjoint connected Eulerian subgraphs $B_1$, $B_2$, $\ldots$, $B_n$ with
  $n\ge 2$
  satisfying (i), (iii), and (iv)
  so that $n$ is maximized.
  Such a choice must exist because the  sequence $G_1$, $G_2$ satisfies
  (i), (iii), and (iv).
  Let $B_{n+1}=B_1$ and $B_0=B_n$.
  Note that $\abs{V(B_i)}>1$ because otherwise
  either $V(B_{i-1})\cap
  V(B_{i+1})\neq\emptyset$ when $n>2$
  or $\abs{V(B_1)\cap V(B_2)}\le 1$ when $n=2$.

  Suppose that (ii) is false.
  By symmetry, we may assume that $\abs{V(B_1)}\ge 3$ and
  $B_1$ has a $1$-separation $(F_1,F_2)$ such that both $F_1$ and
  $F_2$ have at least two vertices.
  Let $v$ be the vertex on the boundary of $(F_1,F_2)$.

  Suppose that $n=2$.
  Since $v$ is not a cut vertex of $G$,
  $v\notin V(B_1)\cap V(B_2)$
  and $\abs{V(B_2)\cap V(F_1)}=\abs{V(B_2)\cap V(F_2)}=1$.
  Since $B_1$ is an Eulerian subgraph,
  both $F_1$ and $F_2$ are Eulerian.
  Then a sequence $F_1$, $F_2$, $B_2$ satisfies (i), (iii), and (iv) and
  therefore it contradicts our assumption that $n$ is maximum.

  Thus $n>2$.
  If $F_2$ meets both $B_n$ and $B_2$, then
  $v$ is a cut vertex of $G$, contradicting the assumption that $G$ is
  $2$-connected.
  Thus $F_2$ meets at most one of $B_n$ and $B_2$.
  Similarly $F_1$ meets at most one of $B_n$ and $B_2$.
  Since each $B_i$ is Eulerian, $(F_1,F_2\cup \bigcup_{i=2}^n B_i)$ is an
  even $2$-separation and therefore $F_1$ is Eulerian.
  Similarly $F_2$ is Eulerian because $(F_2,F_1\cup \bigcup_{i=2}^n B_i)$
  is an even $2$-separation.

  We may assume that $F_1$ meets $B_n$ and $F_2$ meets $B_2$.
  Then we consider a sequence
  $F_1$, $F_2$, $B_2$, $\ldots$, $B_n$ of edge-disjoint connected
  Eulerian subgraphs satisfying (i), (iii), and (iv). This contradicts the assumption that $n$ is
  chosen to be maximum.
\end{proof}

We call any $B_1, \dots, B_n$ given by Lemma~\ref{bead} a
\emph{necklace decomposition} of $(G, \Sigma)$ (See
Figure~\ref{fig:beads}).
Each $B_i$ is called a \emph{bead} of the necklace decomposition.

We now proceed to prove a sequence of lemmas concerning a
counterexample $(G, \Sigma)$ to Theorem~\ref{thm:main-signed} with $\abs{E(G)}$ minimum.

  \begin{LEM}\label{odd2sep}
    A minimum counterexample $(G,\Sigma)$ cannot have an odd $2$-separation $(G_1, G_2)$ where each
    $(G_i,E(G_i)\cap \Sigma)$ is connected and not a signed subgraph
    of $\tilde{K_{2}^{2}}$.
  \end{LEM}
  \begin{proof}

    Suppose that $(G,\Sigma)$ has an odd  $2$-separation   $(G_1, G_2)$ such that each $(G_i,E(G_i)\cap \Sigma)$ is connected and not a signed subgraph
    of $\tilde{K_{2}^{2}}$.
    Let $u$ and $v$ be the vertices on the
  boundary of $(G_1,G_2)$.  By assumption, $\deg_{G_1} (u)$ and $\deg_{G_1} (v)$ are odd.

We first handle the subcase that one of the $(G_i,E(G_i)\cap \Sigma)$,
say $(G_2,E(G_2)\cap \Sigma)$ is
balanced.  In this case, by Lemma~\ref{bipartite}, we may assume that
all edges in $G_2$ are positive
and thus $\Sigma\subseteq E(G_1)$.
Let $(G_1',\Sigma)$ be the signed
graph obtained from $(G_1,\Sigma)$ by adding a positive  edge between $u$
and $v$.  Note that $G_1'$ is $2$-connected since $G$ is
$2$-connected.  Moreover, $(G_1',\Sigma)$ is a proper minor of
$(G,\Sigma)$ since there is a balanced path between $u$ and $v$ in
$(G_2,\emptyset)$.
(Note that every path in $(G_2,\emptyset)$ is balanced because no edge is negative.)
By assumption, $(G_1',\Sigma)$ is not a counterexample and therefore
$(G_1',\Sigma)$ has a balanced-cycle decomposition.  It follows that
$E(G_1)$ can be decomposed as $\mathcal{C}_1 \cup \{P_1\}$, where
$\mathcal{C}_1$ is a family of balanced cycles and $P_1$ is a balanced path
between $u$ and $v$.  Since $(G_2,\emptyset)$ has no negative edges, and
$u$ and $v$ are the only odd degree vertices in $G_2$,
we can decompose $E(G_2)$ as $\mathcal{C}_2 \cup \{P_2\}$ where $\mathcal{C}_2$ is a family of balanced cycles, and $P_2$ is a balanced path between $u$ and $v$.  Therefore $\mathcal{C}_1 \cup \mathcal{C}_2 \cup \{P_1 \cup P_2\}$ is a balanced-cycle decomposition of $(G, \Sigma)$.

The other subcase is if neither $(G_1,E(G_1)\cap \Sigma)$ nor
$(G_2,E(G_2)\cap \Sigma)$ is balanced.  Let
$(G_i',\Sigma_i)$ be the signed graph obtained from $(G_i,E(G_i)\cap \Sigma)$ by adding a positive edge
between $u$ and $v$ if $\abs{\Sigma \cap E(G_i)}$ is even, and adding a
negative edge between $u$ and $v$ if $\abs{\Sigma \cap E(G_i)}$ is odd.  Since
$(G_1,E(G_1)\cap \Sigma)$ and $(G_2,E(G_2)\cap \Sigma)$ are
not balanced and $G$ is $2$-connected, by Menger's theorem
we can find two vertex-disjoint paths from $u$ and $v$ to an unbalanced cycle in
$(G_i,E(G_i)\cap\Sigma)$ for each $i\in \{1,2\}$. Therefore, both $(G_1',\Sigma_1)$ and $(G_2',\Sigma_2)$ are proper minors of $(G,
\Sigma)$.  By the minimality assumption, $E(G_i)$ can be decomposed into
$\mathcal{C}_i \cup \{P_i\}$ where $\mathcal{C}_i$ is a family of balanced
cycles and $P_i$ is a path between $u$ and $v$.  Again, $\mathcal{C}_1
\cup \mathcal{C}_2 \cup \{P_1 \cup P_2\}$ is a balanced-cycle decomposition of $(G, \Sigma)$.
\end{proof}

We call a signed subgraph $(H, \Gamma)$ of $(G, \Sigma)$ an \emph{albatross} if $(H, \Gamma) \cong \tilde K_2^2 \cdot \tilde K_2^2$
(see Figure~\ref{fig:k32}) and the degree-$4$ vertex in $H$ also has degree $4$ in $G$.

\begin{LEM} \label{albatrosses}
In a minimum counterexample $(G, \Sigma)$, all albatrosses are edge-disjoint.  Moreover, if $G$ has maximum degree $4$, then all albatrosses are vertex-disjoint.
\end{LEM}
\begin{figure}
\centering
\tikzstyle{every node}=[circle,draw,fill=black!50,inner
  sep=0pt,minimum width=4pt]
  \begin{tikzpicture}
    \draw [bend left] (-1,0) node (v1) {} to (0,0) node (v2) {} to (1,0) node (v3) {} to (2,0) node (v4) {};
    \draw [dashed,bend right] (v1) to (v2) to (v3) to (v4);
    \end{tikzpicture}
\caption{$\tilde K_2^2 \cdot \tilde K_2^2 \cdot \tilde K_2^2$. (Solid
  lines denote negative edges.)}
\label{fig:k23}
\end{figure}
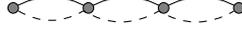
\begin{proof}
Towards a contradiction, let $(H_1, \Sigma_1)$ and $(H_2, \Sigma_2)$ be distinct albatrosses of $(G, \Sigma)$ that are not edge-disjoint. Since $G$ is $2$-connected, it is not possible that $(H_1 \cup H_2, \Sigma_1 \cup \Sigma_2) \cong \tilde K_{3}^2$ (see Figure~\ref{fig:k32}). Thus $(H_1 \cup H_2, \Sigma_1 \cup \Sigma_2) \cong \tilde K_2^2 \cdot \tilde K_2^2 \cdot \tilde K_2^2$ (see Figure~\ref{fig:k23}) and  $H_1 \cup H_2$ attaches to the rest of $G$ at the degree-$2$ vertices of $H_1 \cup H_2$.  Let $(G', \Sigma')$ be the signed graph obtained from $(G, \Sigma)$ by replacing $(H_1 \cup H_2, \Sigma_1 \cup \Sigma_2)$ with a single $\tilde K_2^2$.  Note that $(G', \Sigma')$ is a $2$-connected proper minor of $(G, \Sigma)$.   By the minimality assumption, we have that $(G', \Sigma')$ has a balanced-cycle decomposition $\mathcal{C}'$, which we can lift to a balanced-cycle decomposition $\mathcal{C}$ of $(G, \Sigma)$.  The second part of the lemma follows in the same way.
\end{proof}

We call $(G, \Sigma)$ a \emph{Bermuda triangle} if all necklace decompositions $B_1, \dots, B_n$ of $(G, \Sigma)$ satisfy $n=3$ with two $(B_i, E(B_i) \cap \Sigma)$ isomorphic to $\tilde K_2^2$, see Figure~\ref{bermuda}.  Note that every Bermuda triangle contains an albatross.
\begin{figure}
\centering
\tikzstyle{v}=[circle,draw,fill=black!50,inner
  sep=0pt,minimum width=4pt]
\begin{tikzpicture}
\draw [bend left] (0,1) node [v] (v1) [label=above: $u$]{} to (1,0) node[v] (v2) {}to (0,-1) node[v] (v3) [label=below: $v$]{};
\draw [dashed, bend right] (v1) to  (v2)to (v3);
\fill [black!30,  circular glow={fill=black!10}]  (v1)  to [out=140,in=90]  (-1,0)  to
[out=-90,in=220] (v3) to [out=120,in=220] (v1);
\node  at (-.75,0) {$B_1$};
\end{tikzpicture}
\caption{A \emph{Bermuda triangle}. Note that the right half is an albatross. (Solid lines denote negative edges.)} 
\end{figure}
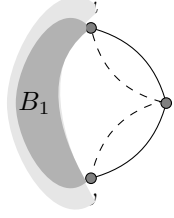

\begin{LEM}\label{even2sep}
    A minimum counterexample $(G,\Sigma)$ cannot have an even
    $2$-separation $(G_1, G_2)$
    such that each
    $(G_i,E(G_i)\cap \Sigma)$ is connected and
    not isomorphic to
    $\tilde{K_{2}^{2}}$,
    unless $(G, \Sigma)$ is a Bermuda triangle.
\end{LEM}
\begin{proof}
  Suppose that $(G,\Sigma)$ has an even $2$-separation   $(G_1, G_2)$
  such that each $(G_i,E(G_i)\cap \Sigma)$
  is connected and
  not isomorphic to $\tilde{K_{2}^{2}}$  and that $(G, \Sigma)$ is not a Bermuda
  triangle.
  By Lemma~\ref{bead}, $(G, \Sigma)$ has a necklace decomposition
  $B_1, \dots, B_n$.  If each bead $B_i$ contains an even number of
  negative edges, we contradict the fact that $(G, \Sigma)$ is a minimum
  counterexample.  Therefore, there are at least two beads with an odd
  number of negative
  edges.  Call such a bead an \emph{odd bead}.  Moreover, since $(G,
  \Sigma)$ is not a Bermuda triangle, we may assume that $n \neq 3$ or
  $n=3$ and at most one $(B_i, E(B_i) \cap \Sigma) \cong \tilde
  K_2^2$.  It follows that we may choose an even $2$-separation $(H_1,
  H_2)$ of $(G, \Sigma)$
  such that both  $(H_1,E(H_1)\cap \Sigma)$ and $(H_2,E(H_2)\cap \Sigma)$
  contain an odd bead and are not %
  isomorphic to $\tilde{K}_2^2$.
  This also implies that $\abs{E(H_1)}, \abs{E(H_2)}>2$ because odd beads
  that are not isomorphic to $\tilde{K}_2^2$ contain at least three edges.

Let $x$ and $y$ be the vertices on the boundary of $(H_1, H_2)$.  We
first handle the subcase that $H_1$ and $H_2$ contain an odd number of
negative edges.  In this case we define $(H_i',\Sigma_i)$ to be the signed
graph obtained from $(H_i,E(H_i)\cap \Sigma)$ by adding
 a positive edge and a negative edge between $x$ and $y$.  Since each $H_i$
 contains an odd bead, there exist an unbalanced $x$-$y$-path and a balanced
 $x$-$y$-path in $(H_i,E(H_i)\cap \Sigma)$.  Therefore,
 $(H_i',\Sigma_i)$ is odd-$K_4$-minor-free and $2$-connected.  By the
 minimality assumption, $E(H_i)$ can be decomposed as $\mathcal{C}_i
 \cup \{E_i\} \cup \{O_i\}$, where $\mathcal{C}_i$ is a family of balanced
 cycles, $E_i$ is a balanced $x$-$y$-path and $O_i$ is an unbalanced
 $x$-$y$-path.  But then $\mathcal{C}_1 \cup \mathcal{C}_2 \cup \{E_1
 \cup E_2\} \cup \{O_1 \cup O_2\}$ is a balanced-cycle decomposition of
 $(G, \Sigma)$.

 The other remaining subcase is when
 $(H_1,E(H_1)\cap\Sigma)$ and $(H_2,E(H_2)\cap \Sigma)$
 contain an even number of negative edges.  In this case we let $(H_i',\Sigma_i)$ be
 the graph obtained from $(H_i,E(H_i)\cap \Sigma)$ by adding two positive edges between $x$
 and $y$. Again note that $(H_i',\Sigma_i)$ is $2$-connected and
 odd-$K_4$-minor-free.  By the minimality assumption, $(H_i',\Sigma_i)$ is balanced-cycle
 decomposable.  Moreover, since $(H_i,E(H_i)\cap \Sigma)$ is \emph{not} balanced-cycle
 decomposable (by virtue of containing an odd bead), the balanced
 $2$-cycle formed by the two newly
 added edges cannot be used as a cycle in the decomposition.
Therefore, $E(H_i)$ can be decomposed as $\mathcal{C}_i \cup \{E_{i,1}
,E_{i,2} \}$ where $\mathcal{C}_i$ is a family of balanced cycles
and $E_{i,1}$ and $E_{i,2}$ are balanced $x$-$y$-paths.    But then
$\mathcal{C}_1 \cup \mathcal{C}_2 \cup \{E_{1,1} \cup E_{2,1},
E_{1,2} \cup E_{2,2}\}$ is a balanced-cycle decomposition of $(G,
\Sigma)$.
\end{proof}

\begin{LEM}\label{del2-cycle}
  A minimum counterexample $(G,\Sigma)$ cannot contain two parallel
  edges of the same sign.
\end{LEM}
\begin{proof}
  Let $G_2$ be a connected subgraph of $G$ having exactly two parallel
  edges of the same sign.
  Let $G_1=G\setminus E(G_2)$. Since $G$ is $2$-connected, $G_1$ is connected.
  By Lemma~\ref{bead}, there is a necklace decomposition
  $B_1,B_2,\ldots,B_n$ extending the even $2$-separation $(G_1, G_2)$.
  We may assume that $B_n=G_2$.
  Since $\abs{\Sigma}$ is even, $G_1$ is not isomorphic to $\tilde K_2^2$.
  Lemma~\ref{even2sep} implies that $G$ is a Bermuda triangle
  and therefore $n=3$ and $B_1$, $B_2$ are
  isomorphic to $\tilde K_2^2$.  But then, $(G,\Sigma)$ is easily seen to be decomposable
  into two balanced $3$-cycles.
\end{proof}
We say that a signed graph $(G, \Sigma)$ is \emph{almost
  $3$-connected} if for all proper $2$-separations $(G_1, G_2)$ of $G$,
$(G_1, E(G_1) \cap \Sigma)$ or $(G_2, E(G_2) \cap \Sigma)$ is
isomorphic to $\tilde K_2^2 \cdot \tilde K_2^2$.

\begin{LEM} \label{bermudaconnected}
  A minimum counterexample $(G,\Sigma)$ is almost $3$-connected.
\end{LEM}

\begin{proof}
  Let $(G_1, G_2)$ be a proper $2$-separation of $G$.
  Since $G$ is $2$-connected, $\abs{E(G_1)}, \abs{E(G_2)}>2$
  and $G_1$, $G_2$ are connected.
  By Lemma~\ref{odd2sep}, $(G_1, G_2)$ must be an
  even $2$-separation.
  By Lemma~\ref{even2sep}, $G$ is a Bermuda triangle.
  Let $B_1,B_2,B_3$ be a necklace decomposition of $G$ extending $(G_1,G_2)$
  given by Lemma~\ref{bead}.
  We may assume that $B_1$, $B_2$ are isomorphic to $\tilde K_2^2$
  and $G_1=B_1\cup B_2$.
  This implies that $(G_1,E(G_1)\cap \Sigma)$ is isomorphic to $\tilde
  K_2^2\cdot \tilde K_2^2$.
\end{proof}
  
\begin{LEM} \label{3-sep}
  A minimum counterexample $(G, \Sigma)$ cannot have a $3$-separation $(G_1, G_2)$, where $(G_2,E(G_2)\cap \Sigma)$ is
  balanced, connected, and has at least $4$ edges.
\end{LEM}

\begin{proof}
Choose such a $3$-separation $(G_1, G_2)$ with $\abs{E(G_2)}$ minimum.  By Lemma~\ref{bipartite}, we may assume
that all edges in $(G_2,E(G_2)\cap \Sigma)$ are positive and so
$\Sigma\subseteq E(G_1)$.  In particular, all paths contained in $G_2$
are balanced
and $\tilde K_2^2\cdot \tilde K_2^2$ is not a signed
subgraph of $(G_2,\emptyset)$.  Also note that $G_2$ does not contain parallel edges by Lemma~\ref{del2-cycle}. 
Let $x, y$ and $z$ be the vertices on the
boundary of $(G_1, G_2)$.

We first prove that $G_2$ is $2$-connected.  Suppose not and let
$(H_1,H_2)$ be a proper $1$-separation of $G_2$ with $V(H_1) \cap
V(H_2)=\{w\}$.  If $w \in \{x,y,z\}$ then $H_1$ and $H_2$ both induce
$2$-separations in $G$.  Since $(G, \Sigma)$ is almost $3$-connected and
contains no parallel edges of the same sign, it follows that $H_1$ and
$H_2$ are each just a single edge.  This contradicts $\abs{E(G_2)} \geq
4$.  Thus, $w \notin \{x,y,z\}$ and we may assume $V(H_1) \cap
\{x,y,z\}=\{x\}$.  Thus, $H_1$ induces a $2$-separation in $G$ and
must just be the single edge $xw$.  Hence $G_2 - x$ induces a
$3$-separation in $G$.  This contradicts the minimality of $G_2$
unless, $G_2- x$ has exactly three edges.  On the other hand, recall that $G_2$ contains no parallel edges.   Furthermore, Lemma~\ref{bermudaconnected} implies that no vertex of $G$ has degree $2$ (such a vertex gives a proper $2$-separation with neither side equal to $\tilde K_2^2 \cdot \tilde K_2^2$).  Since $G$ is Eulerian, $G$ has minimum degree at least $4$.  Since $G_2$ contains only four edges, we conclude that $V(G_2) \setminus \{x,y,z\}=\{w\}$.  Since $\deg_{G_2}(w) \geq 4$, $G_2$ must contain parallel edges, which is a contradiciton.

Now by Lemma~\ref{parity}, \[\deg_{G_1}(x)+\deg_{G_1}(y)+\deg_{G_1}(z)\equiv 0\pmod2.\]
There are two possibilities to consider: either two of $x,y$, and $z$
have odd degree in $G_1$ or none of $x,y$, and $z$ have odd degree in
$G_1$.

We handle the former possibility first.  By symmetry, suppose that $\deg_{G_1}(x)$ and $\deg_{G_1}(y)$ are odd.
Let $G_1^{e,f}$ be the graph
obtained from $G_1$ by
adding an edge $e$ between $z$ and $y$ and an edge $f$ between $z$ and $x$.  We claim that
$G_1^{e,f}$ is $2$-connected.
Suppose not and let $(H_1, H_2)$ be a proper $1$-separation of $G_1^{e,f}$, with $V(H_1) \cap V(H_2)=\{w\}$.
Note that $\{x,y,z\}$ cannot be a subset of $V(H_i)$, else $w$ is a cut-vertex of $G$.
Now, if $z \in V(H_i) \setminus \{w\}$, then $\{x,y\} \subseteq V(H_i)$, since $zx$ and $zy$ are edges of $G_1^{e,f}$, a contradiction.
Hence $w=z$.
By symmetry we may assume $x \in V(H_1)$ and $y \in V(H_2)$.  Thus, $H_1 \setminus f$ and $H_2 \setminus e$ induce
$2$-separations in $G$, with boundary vertices $\{x,z\}$ and $\{y,z\}$ respectively.
Each of these $2$-separations is odd since $\deg_{H_1 \setminus f}(x)=\deg_{G_1}(x)$ and
$\deg_{H_2 \setminus e}(y)=\deg_{G_1}(y)$.  By Lemma~\ref{odd2sep}, $H_1 \setminus f$ and $H_2 \setminus e$ are
each just a single edge.  But now, $G -z$ is a subgraph of $G_2$, and is hence balanced as a signed graph. This contradicts Proposition~\ref{almostbipartite}.  Thus, $G_1^{e,f}$ is $2$-connected as claimed.

 Since $G_2$ is also $2$-connected, there are two internally disjoint paths $P_1$ and $P_2$ in $G_2$ from $\{z\}$ to $\{x,y\}$.
 Let $H:=G_1 \cup E(P_1) \cup E(P_2)$.  Observe that $H$ is $2$-connected since it is a subdivision of
 $G_1^{e,f}$. Evidently, $H$ is Eulerian and contains
an even number of negative edges since $\Sigma \cap E(G_2)=\emptyset$.
Moreover, $P_1 \cup P_2 \neq G_2$ since $G$ has minimum degree $4$ by Lemma~\ref{bermudaconnected}.  Hence $H \neq G$.
We are done since $(H, \Sigma \cap E(H))$ is balanced-cycle decomposable by induction and
$(G \setminus E(H), \Sigma \setminus E(H))$ is balanced-cycle decomposable since it is balanced.

We now consider the second possibility that each of $\deg_{G_1}(x), \deg_{G_1}
(y)$, and $\deg_{G_1}(z)$ is even.  In this case we let $(G_1^\Delta, \Sigma)$ be the signed graph obtained from $(G_1, \Sigma)$ by adding three
positive edges $e=xy, f=yz$, and $g=xz$.  Evidently, $G_1^\Delta$ is Eulerian,
$2$-connected and contains an even number of negative edges.

We claim that $(G_1^\Delta,\Sigma)$ is a minor of
$(G,\Sigma)$.  First observe that $G_2$ contains a cycle $C$ as it is Eulerian.
Recall that $G_2$ contains no parallel edges; hence $C$ is not a $2$-cycle.
Next notice that $G$ has no proper $2$-separation $(H_1,H_2)$ with
$V(H_1)\cap V(H_2)\subseteq V(G_2)$ because $G$ is almost
$3$-connected and $G_2$ has no negative edges.
By Menger's theorem, there are three
vertex-disjoint paths from $\{x,y,z\}$ to $V(C)$.
Now it is easy to obtain $(G_1^\Delta,\Sigma)$ as a minor of
$(G,\Sigma)$; we begin by contracting
edges in those three paths.

By the minimality assumption, $(G_1^\Delta, \Sigma)$ has a balanced-cycle decomposition $\mathcal{C}^\Delta$.  If $\{e,f,g\} \in
\mathcal{C}^\Delta$, then $(G_1, \Sigma)$ is balanced-cycle decomposable.
But $(G_2, \emptyset)$ is also balanced-cycle decomposable since it is
Eulerian and balanced.  Thus, $(G, \Sigma)$ is balanced-cycle decomposable.

If $e \in C_1 \in \mathcal{C}^\Delta$ and $\{f,g\}
\subset C_2 \in \mathcal{C}^\Delta$ for $C_1 \neq C_2$, then it
suffices to find two edge-disjoint $x$-$y$ paths in $G_2$, at least
one of which avoids $z$. Since $G_2$ is $2$-connected, $G -z$ is connected, so there does
exist an $x$-$y$ path $P$ in $G_2$ that avoids $z$.  But now the
second path exists since $x$ and $y$ are the only odd degree vertices
in $G_2 \setminus E(P)$.

By symmetry, the only remaining possibility is if
$e,f$, and $g$ are in different cycles of $\mathcal{C}^\Delta$.  In
this case it suffices to show that there are pairwise edge-disjoint
paths $Q_{x,y}, Q_{y,z}$, and $Q_{x,z}$, where $Q_{i,j}$ is an
$i$-$j$-path in $G_2$ such that $\abs{V(Q_{i,j}) \cap \{x,y,z\}}=2$.
We may assume that $G_2$ has no cycle containing $x$, $y$, and $z$, else we are done.
Since $G_2$ is $2$-connected, $G_2$ has a cycle $C$ containing $y$ and
$z$.
Since $G_2$ has no cycles containing $x$, $y$, and $z$,
there do not exist three vertex-disjoint paths from the neighbours of $x$ in $G_2$ to
$V(C)$.
By Menger's theorem, there is a proper $2$-separation $(H_1, H_2)$ of
$G_2$ such that $x\in V(H_1)\setminus V(H_2)$ and $y,z\in
V(C)\subseteq V(H_2)$.
Since $H_1$
induces a $3$-separation in $G$, we have that $\abs{E(H_1)} \leq 3$.  In
particular, $x$ has degree $2$ in $H_1$ (and hence also in $G_2$).
Since $G_2$ is $2$-connected, we can find two vertex-disjoint paths
$Q_{x,y}$ and $Q_{x,z}$.
Note
that in $G_2 \setminus E(Q_{x,y} \cup Q_{x,z})$, $y$ and $z$ are the
only vertices of odd degree and $x$ is an isolated vertex.  Thus,
$G_2\setminus E(Q_{x,y}\cup Q_{x,z})$ has a path
$Q_{y,z}$ from $y$ to $z$ avoiding $x$.  This completes the proof.
\end{proof}

We are now ready to prove our main theorem.

\begin{THMMAIN}
Every $2$-connected Eulerian loopless odd-$K_4$-minor-free signed graph with an even number of negative edges is balanced-cycle decomposable.
\end{THMMAIN}
\begin{proof}
  Suppose not.
  Let $(G,\Sigma)$ be a counterexample with
$\abs{E(G)}$ minimum.  By Lemma~\ref{del2-cycle}, there are at most two parallel edges
  between every pair of vertices.  If $x$ is a degree-$2$ vertex, we can  
  contract an edge incident with $x$, if necessary by re-signing, to obtain a smaller counterexample.  
  Thus, $G$ has minimum degree at least $4$.

  By Theorem~\ref{structure}, one of the following holds.

  \begin{enumerate}[(i)]
  \item
 $(G, \Sigma)$ is almost balanced or planar with two unbalanced faces.

  \item
    $(G, \Sigma)$ has a $2$-separation $(G_1, G_2)$ where each
    $(G_i,E(G_i)\cap \Sigma)$ is connected and not a signed subgraph of $\tilde{K_{2}^{2}}$.

  \item
    $(G, \Sigma)$ has a $3$-separation $(G_1, G_2)$ where
    $(G_2,E(G_2)\cap \Sigma)$ is balanced, connected, and has at least $4$ edges.
  \end{enumerate}

  By  Propositions~\ref{almostbipartite} and \ref{oddplanar}, (i) is not possible.
 Lemma~\ref{3-sep} implies that (iii) is impossible.

 Thus (ii) holds and let $(G_1, G_2)$ be such a
 $2$-separation.
By Lemma~\ref{odd2sep}, $(G_1, G_2)$ is an even separation.  By
Lemma~\ref{even2sep}, $(G, \Sigma)$ is a Bermuda triangle.  Let $B_1,
B_2, B_3$ be a necklace decomposition of $G$ with $(B_2, E(B_2) \cap
\Sigma) \cong (B_3, E(B_3) \cap \Sigma) \cong \tilde K_2^2$.
If $B_1$ has only two vertices, then $B_1$ is just two edges of the same sign, which contradicts Lemma~\ref{del2-cycle}.
So $B_1$ is $2$-connected.
Since $B_1$ contains an even number of negative edges, by the minimality assumption,
 $B_1$ has a balanced-cycle decomposition~$\mathcal{C}$.

 We first claim that $G$ is $4$-regular.
Let $u$ and $v$ be the vertices in $V(B_1) \cap (V(B_2) \cup V(B_3))$.
 Choose a shortest chain of balanced cycles $C_1$, $C_2$, $\ldots$,
$C_k$  ($k\ge 1$)
in $\mathcal C$ such that
$u\in V(C_1)$, $v\in V(C_k)$
and
$V(C_i)\cap V(C_{i+1})\neq \emptyset$ for all $i\in\{1,2,\ldots,k-1\}$.
Since this chain is shortest, $V(C_i)\cap V(C_j)=\emptyset$ whenever
$j>i+1$.
Let $W=E(B_1)-\bigcup_{i=1}^k E(C_i)$.
If $W$ is nonempty, then
$\mathcal C-\{C_1,C_2,\ldots,C_k\}$ is a balanced-cycle decomposition of
$(B_1-\bigcup_{i=1}^k E(C_i), W\cap \Sigma)$.
Moreover by the minimality assumption,
the signed graph $(G_0,\Sigma_0)=(G-W, (E(G)-W)\cap \Sigma)$ must have
a balanced-cycle decomposition $\mathcal D$.
Then $(\mathcal C-\{C_1,C_2,\ldots,C_k\})\cup \mathcal D$ is a balanced-cycle decomposition of $(G,\Sigma)$, which is a contradiction. Therefore $W$ is empty and hence $B_1 = \bigcup_{i=1}^k
C_i$.  Consequently, every vertex of $G$ has degree at most $4$ (and
hence exactly $4$).
Since $\abs{V(B_1)}>2$, we have $k>1$ and $u$ is not adjacent to $v$ in $G$.

Let $(G', \Sigma')$ be the signed graph obtained from $(G, \Sigma)$ by replacing each albatross with a $\tilde K_2^2$.  Since all albatrosses are vertex-disjoint by Lemma~\ref{albatrosses}, $(G', \Sigma')$ is well-defined.  By applying Theorem~\ref{structure} again we conclude that one of the following holds.
   \begin{enumerate}
  \item
 $(G', \Sigma')$ is almost balanced.

 \item
 $(G', \Sigma')$ is planar with two unbalanced faces.

  \item
    $(G', \Sigma')$ has a $2$-separation $(G_1', G_2')$ where each
    $(G_i',E(G_i')\cap \Sigma')$ is connected and not a signed subgraph of $\tilde{K_{2}^{2}}$.

  \item
    $(G', \Sigma')$ has a $3$-separation $(G_1', G_2')$ where
    $(G_2',E(G_2')\cap \Sigma')$ is balanced, connected, and has at least $4$ edges.
  \end{enumerate}

By Lemma~\ref{bermudaconnected}, $(G, \Sigma)$ is almost $3$-connected,
and hence $(G', \Sigma')$ is $3$-connected. Let $x,y,z$ be the vertices of an arbitrary albatross $A$ in $(G, \Sigma)$, where $\deg_A(y)=4$.  Note that $x$ and $z$ are not adjacent in $G$  
since $(G, \Sigma)$ is almost $3$-connected.
Therefore, since $(G, \Sigma)$ does not contain any
parallel edges of the same sign, $(G', \Sigma')$ also does not contain any
parallel edges of the same sign.  It follows that (3) is impossible.

We next exclude possibility (1).  By re-signing in $G'$, we may assume that
there exists a vertex $t \in V(G')$ such that all edges of $G'$ are
positive except possibly those incident with $t$.  Since all albatrosses of $(G, \Sigma)$ are vertex-disjoint, it follows that $B_2 \cup B_3$ is the only albatross of $(G, \Sigma)$, and that $t \in \{u,v\}$. Thus, $(G', \Sigma')$ is obtained from $(G, \Sigma)$ by replacing $B_2 \cup B_3$ with a $\tilde K_2^2$ between $u$ and $v$.  Let $w$ be the vertex of $B_2 \cup B_3$ not in $\{u,v\}$.  By performing the same re-signings in $G$ as we performed in $G'$ we may assume that all edges in $G$ are positive except possibly those incident with $t$ or $w$.

As $G$ is $4$-regular and contains an even number of
negative edges, exactly $1$ or $3$ edges incident with $t$ are negative.  By
re-signing at $t$ in $G$, we may assume that exactly one edge incident with $t$ is
negative. This negative edge is also necessarily incident with $w$.  Therefore, $G-w$ is balanced as a signed graph, which contradicts Proposition~\ref{almostbipartite}.

We next handle possibility (4).  Let $(G_1', G_2')$ be such a
$3$-separation.  Since $(G_2', E(G_2') \cap \Sigma')$ is balanced, it
evidently cannot contain any $\tilde K_2^2$ subgraphs.  Thus, by
putting back the albatrosses, we obtain a $3$-separation $(G_1, G_2)$ of $(G,
\Sigma)$ where $(G_2,E(G_2)\cap \Sigma)$ is balanced, connected, and
has at least $4$ edges.
This is impossible by Lemma~\ref{3-sep}.

We finish by ruling out possibility (2). 
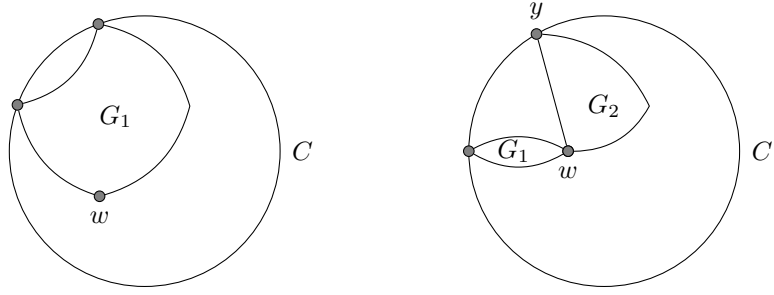
\begin{figure}
\begin{tikzpicture}[scale=.60]
\tikzstyle{v}=[circle,draw,fill=black!50,inner
  sep=0pt,minimum width=4pt,solid,thin]
\draw  (0:0)  circle (3);
\node at (130:1) {$G_1$};
\node at (0:3.5) {$C$};
\node [v] (v1) at (110:3) {};
\node [v] (v2) at (160:3) {};
\draw [bend left](v1) to (v2);
\draw [bend right] (v2)to(-1,-1)node[v,label=below:$w$] {}to (1,1) to (v1);
\end{tikzpicture}
\hspace{5em}%
\begin{tikzpicture}[scale=.60]
\tikzstyle{v}=[circle,draw,fill=black!50,inner
  sep=0pt,minimum width=4pt,solid,thin]
  \node[v] at (180:.8) [label=below: $w$]{};
\draw  (0:0)  circle (3);
\draw[bend left] (180:3) node[v] (v2) {}to(180:.8) node[v] (v1) {};
\draw[bend right] (v2)to(v1);
\node  at (0,1) {$G_2$};
\node at (-2,0) {$G_1$};
\node at (0:3.5) {$C$};
\node [v] (v3) at (120:3)[label=above: $y$] {};
\draw (v1) to (v3)[bend left] to (1,1) to (v1);
\end{tikzpicture}
\caption{On the left, $G_1\not \cong \tilde K_2^2$ and shares at least two vertices
with $C$. On the right, $G_1\cong \tilde K_2^2$ and shares at least one vertex with $C$.}
\label{fig:planar}
\end{figure}
In this case, we claim $(G', \Sigma')$ has a balanced Eulerian signed subgraph $(H', \Sigma'\cap E(H'))$ such that
$E(H')\not = \emptyset$ and $G'\setminus E(H')$ is $2$-connected after removing isolated vertices.
We can then lift $(H', \Sigma' \cap E(H'))$ to a balanced Eulerian signed subgraph $(H, \Sigma\cap E(H))$ of $(G, \Sigma)$ as follows.  Suppose that an albatross $A$ in $(G, \Sigma)$ has been replaced with $K=\tilde K_2^2$ in $(G', \Sigma')$.  Note that $H'$ cannot use both edges of $K$, since $(H', \Sigma' \cap E(H')$ is balanced.  If $H'$ uses a positive edge $e$ of $K$, we replace $e$ by a balanced path in $A$.  If $H'$ uses a negative edge $e$ of $K$, we replace $e$ by an unbalanced path in $A$.  If $H'$ uses no edges of $K$, then $H$ uses no edge of $A$.  

Note that $G \setminus E(H)$ is $2$-connected after removing isolating vertices, since $G' \setminus E(H')$ is $2$-connected after
removing isolated vertices. By Lemma~\ref{bipartite}, $H$ (and hence $G \setminus E(H)$) has
an even number of negative edges. We are then finished since $(G \setminus E(H), \Sigma \setminus E(H))$
is balanced-cycle decomposable by induction.  We now proceed to show that such an $H'$ exists.

We again avoid using Theorem~\ref{seymour} to keep
our proof self-contained. Consider a fixed planar embedding of $(G', \Sigma')$ with at
most two unbalanced faces. Note that each $\tilde{K}_2^2$ must bound a face of $(G', \Sigma')$,
else its two vertices would form a vertex-cut of size two. Since $(G', \Sigma')$ contains
no parallel edges of the same sign, it follows that $G'$ contains at most two pairs
of parallel edges.

Recall that $u$ and $v$ are the vertices of a contracted albatross in $(G', \Sigma')$.
Let $F$ and $F'$ be the two faces adjacent to the unbalanced face given by $u$ and $v$. We may assume
$F$ is the outer face and balanced. Let $C$ be the boundary cycle of $F$. Let $G''=G'\setminus E(C)$,
and $\{G_1,...,G_k\}$ be its set of blocks. We may assume $k\ge 2$, else let $H'=C$.
Note that each $G_i$ is an Eulerian plane subgraph of $G'$, and is $2$-connected unless it
is a $\tilde K_2^2$. We may further assume at least one $G_i$, say $G_1$, is not balanced
else let $H'=G''$.

Since each $(G_i, \Sigma' \cap E(G_i))$ inherits all of its finite faces from $(G', \Sigma')$, and $C$
contains an edge from a finite unbalanced face of $(G', \Sigma')$, there is at most one finite unbalanced face
left in $(G'', \Sigma' \cap E(G''))$. Therefore, for $i\geq 2$, every finite face of $(G_i, \Sigma' \cap E(G_i))$ is balanced and
hence $(G_i, \Sigma' \cap E(G_i))$ is balanced. We may assume for $i\geq 2$, $(G_i, \Sigma' \cap E(G_i))$ is a balanced cycle of length at
least three, else let $H'=G_i\setminus E(C_i)$ where $C_i$ is the boundary cycle of the outer face of $G_i$.

Since $G'$ is $3$-connected, every face of $G'$ is bounded by a
non-separating chordless cycle.
Moreover $G_1$ has an unbalanced finite face and therefore no edge of $C$ has
a parallel edge except the edge joining $u$ and  $v$.
As $C$ is non-separating, $G''$ is connected, so its block graph $T$ is a tree. Let $G_i$ be a leaf in $T$
for some $i\not =1$. Let $x$ be the cut vertex of $G''$ belonging to $G_i$.
Since $G_i$ is a cycle, every vertex in $V(G_i) \setminus \{x\}$ must belong to $C$ because $G'$ is $4$-regular.
Since $C$ is chordless, every edge of $G_i - x$ must have endpoints $u$ and $v$.  In particular, $G_i$ must be
a triangle and $G_i$ is the only leaf in $\{G_2, \dots, G_k\}$.  Thus, $T$ is a path with leaves $G_1$ and $G_i$.  We re-label so that the $G_j$'s are labelled
according to their order in $T$.  Let $w$ be the cut-vertex of $G''$ belonging to $G_1$.

Suppose $(G_1, \Sigma' \cap E(G_1)) \not \cong \tilde K_2^2$. Since $G_1$ has at least three vertices and $G'$ is
$3$-connected, $G_1$ must share at least two vertices with $C$, otherwise $(V(G_1)\cap V(C))\cup\{w\}$
is a vertex-cut in $G'$ of size at most two. Therefore $G_1\cup C$ is $2$-connected,
see Figure~\ref{fig:planar}. We are done since we may take $H'=G_2\cup...\cup G_k$.

So we may assume $(G_1, \Sigma' \cap E(G_1)) \cong \tilde K_2^2$.  Since the two cut-vertices of $G''$ in $G_2$ cannot be a
vertex-cut in $G'$, $G_2$ shares at least one vertex $y$ with $C$.  Also, $y \neq w$,
since $G'$ is $4$-regular. Hence, $G_1\cup G_2\cup C$ has an ear decomposition starting from $G_2$ (see Figure~\ref{fig:planar}),
and is thus $2$-connected.  We are finished by letting
$H'=G_3\cup...\cup G_k$, unless $k=2$.  If $k=2$, then $(G_1, \Sigma' \cap E(G_1)) \cong \tilde K_2^2$ and $G_2$ and $C$ are triangles.
Hence $G'$ is obtained from $K_4$ by doubling a pair of independent edges.
Since each set of parallel edges is a $\tilde K_2^2$, there is a balanced $4$-cycle which passes through both sets of parallel edges.  Let
$H'$ be such a balanced $4$-cycle.  We are done since $G'\setminus E(H')$ is just a $4$-cycle and evidently $2$-connected.
\end{proof}

\textbf{Acknowledgements.} The authors would like to thank Cheolwon Heo for permission to include the proof of Theorem~\ref{counterexample}.
The authors would also like to thank anonymous referees for their valuable suggestions.

\end{document}